%% file: FullKR.tex
\documentclass{amsart}

\usepackage{stackengine}

\input{mymacros.tex}

\title{Khovanov-Rozansky homology over $\F[U,V]$}
\author{David Popovi\'c}
\date{June 2026}

\begin{document}
\begin{abstract}
    We define a `full' version of Khovanov-Rozansky homology: a chain complex $\CH(K)$ over $\F[U, V]$ whose chain homotopy type is a knot invariant and which has the property that setting $U = V = 0$ recovers the reduced Khovanov-Rozansky homology. We explore the algebraic structure of $\CH(K)$ and show that a variation of the structure theorem for knot Floer homology applies. This allows us to define counterparts to knot Floer concordance invariants $\tau$, $\epsilon$ and $\phi_j$ in the Khovanov-Rozansky setting.
\end{abstract}
\maketitle

\section{Introduction}
There are two types of modern knot homology theories -- those whose origin can be traced back to gauge theory and those that arose through categorification of quantum groups. Invariants inspired by gauge theory include monopole Floer homology \cite{kronheimer2007monopoles}, Heegaard Floer homology and knot Floer homology \cite{ozsvath2004holomorphic, ozsvath2004holomorphicKnots, rasmussen2003floer}, instanton Floer homology \cite{kronheimer2011knot}, and embedded contact homology \cite{hutchings2010embedded}. They generally admit a clearer geometric interpretation and have been defined for links in arbitrary $3$-manifolds whereas the extension of quantum homology theories such as Khovanov homology \cite{khovanov2000categorification}, Khovanov-Rozansky's $\mathfrak{sl}_n$ homology \cite{khovanov2004matrix} and Khovanov-Rozansky homology \cite{khovanov2008matrix} to this setting remains wide open. Quantum homology theories have also been computed in fewer infinite families of examples and are not understood as well geometrically. Nonetheless, they have been used to prove remarkable results for which no alternative proof is presently known, such as \cite{piccirillo2020conway}.

Due to the completely different origin, construction, and properties, relating different types of invariants tends to be a challenging task. Accordingly, whenever such a connection can be established, it tends to be very exciting. Among the most celebrated results in this direction are the Dowlin's spectral sequence \cite{dowlin2024spectral} from $\overline{Kh}(K)$ to $\widehat{\HFK}(K)$, Ozsv\'ath-Szab\'o's spectral sequence \cite{OS-HFKand4ballGenus} from $\overline{Kh}(L)$ to $\widehat{\HF}(\Sigma_2(L))$ and the recent spectral sequence from $\overline{H}(K)$ to $\widehat{\HFK}(K)$ by Beliakova, Putyra, Robert, and Wagner \cite{beliakova2022proof}.

One structural difference between the Floer theoretic and quantum invariants is that the former admit several different flavors -- $\widehat{\,}$ (hat), $-$ (minus), $+$ (plus), and the full version. The reduced versions of Khovanov-Rozansky $\mathfrak{sl}_n$ and Khovanov-Rozansky homologies are believed to correspond to the hat version while the other flavors have not been uncovered yet.

\vspace{1em}

Motivated by the additional data of higher differentials in the minus version of knot Floer homology, Dunfield, Gukov and Rasmussen \cite[Conjecture 3.1]{dunfield2006superpolynomial} conjectured the existence of a $\Z$-indexed family $d_n$ of higher differentials. Through some slow but steady progress, their question has guided the research in the field for the next two decades. The positive differentials and $d_{-1}$ were constructed by Rasmussen \cite{rasmussen2016some} and the existence of the negative differentials more generally was confirmed after the symmetry of Khovanov-Rozansky homology was established in \cite{oblomkov2020soergel}. The rest of the structure seems to be slightly less like what was envisioned in \cite{dunfield2006superpolynomial}. In particular, it seems that $d_0$ is not a differential and that the family of these operators is more naturally indexed by $\N^2 \setminus \{(0, 0)\}$, as opposed to $\Z$. The positive differentials $d_n$ should be understood as $d_{n|0}$, the negative differentials should be understood as $d_{0|n}$, and $d_0$ fits into this framework as $d_{1|1}$. Based on the main result of \cite{gorsky2024tautological}, the operators $d_{i|j}$ have been constructed recently by Chandler and Gorsky \cite{chandler2024structures}, but it is unknown whether they satisfy $d^2_{i|j} = 0$ in general. Inspired by the similarities with knot Floer homology, we conjecture that this is not always the case. In this paper, we explain how to lift this data of higher differentials $d_{n|0}$ and operators $d_{i|j}$ to a full version of Khovanov-Rozansky homology.

We work over a field $\F$ of characteristic $\mathrm{char}(\F) = 0$.
\begin{theorem}\label{thm:full HOMFLY complex}
Let $K \subset S^3$ be a knot. There is a $\Z\oplus\Z\oplus\Z$-graded finitely generated free chain complex $(\CH(K), \partial)$ over $\F[U,V]$ such that
\begin{enumerate}
    \item the isomorphism type of $\CH(K)$ is a knot invariant, and
    \item $H_*(\CH(K) \otimes_{\F[U,V]} \F) \cong \widehat{\mathit{H}}(K)$ as $\Z\oplus\Z\oplus\Z$-graded vector spaces where $\widehat{\mathit{H}}(K)$ denotes the reduced Khovanov-Rozansky homology.
\end{enumerate}
\end{theorem}
In the rest of the paper, we explore structural properties of the invariant $\CH(K)$. In analogy with knot Floer homology, we prove that the connected sum of knots corresponds to a tensor product of chain complexes.
\begin{theorem}\label{thm:connected sums}
    Let $K_1, K_2 \subset S^3$ be knots. Then
    $$\CH(K_1 \# K_2) \cong \CH(K_1) \otimes_{\F[U,V]} \CH(K_2)$$
    as $\Z\oplus\Z\oplus\Z$-graded chain complexes over $\F[U,V]$.
\end{theorem}
Other algebraic properties of the full Khovanov-Rozansky homology are very similar to those of the full version of knot Floer homology as well. Considering that the knot Floer differential $\partial$ counts $J$-holomorphic disks in some symmetric product of surfaces and the Khovanov-Rozansky differential $\partial$ from Theorem \ref{thm:full HOMFLY complex} is constructed using a combinatorially defined spectral sequence, this is a surprising result that is probably indicative of some deeper connections. For us, it means that the structure theorem of knot Floer homology \cite{popovic2023algebraic} can be used to deduce an analogous result about Khovanov-Rozansky homology of $K$.

\begin{theorem}\label{thm:HOMFLY splitting}
Let $K \subset S^3$ be a knot and let $\CH_{\mathcal{R}_1}(K) = \CH(K) \otimes_{\F[U,V]} \mathcal{R}_1$ be its Khovanov-Rozansky complex over $\mathcal{R}_1 = \frac{\F[U,V]}{(UV)}$. Then
\[\CH_{\mathcal{R}_1}(K) \simeq S_1 \oplus \dots \oplus S_m \oplus L_1 \oplus \dots \oplus L_k\]
where $S_1, \dots, S_m$ are snake complexes and $L_1, \dots, L_k$ are local systems. Moreover, the direct summands are unique up to permutation.
\end{theorem}
This structure allows us to  assign Khovanov-Rozansky counterparts to many concordance invariants arising in knot Floer homology. In particular, we define the Khovanov-Rozansky versions $\tau^H$ of Ozsv\'ath-Szab\'o's $\tau$ invariant \cite{ozsvath2003tau}, $\epsilon^H$ of Hom's $\epsilon$ invariant \cite{hom2014epsilon}, and $\phi^H_j$ of concordance homomorphisms $\phi_j$ \cite{dai2021more}.
\begin{theorem}\label{thm:invariants}
There is a family of knot invariants $\tau^H$, $\epsilon^H$, and $\phi_j^H$ for $j \in \N$. The invariants $\tau^H$ and $\phi^H_j$ for $j \in \N$ are additive under the connected sum of knots.
\end{theorem}
We conjecture that $m=1$ in the statement of Theorem \ref{thm:HOMFLY splitting}, as is the case in knot Floer homology. On the other hand, there are plenty of knots for which the splitting of $\CH(K)$ from Theorem \ref{thm:HOMFLY splitting} is not the same as the corresponding splitting of $\CFK(K)$ in knot Floer homology. It follows that the invariants $\tau^H$, $\epsilon^H$ and $\phi_j^H$ are not the same as $\tau$, $\epsilon$ and $\phi_j$ in general. We give an example in which they differ in Section \ref{sec:example}.
\begin{conjecture}\label{conj:concordance homomorphisms}
Knot invariants $\tau^H$, $\epsilon^H$, and $\phi_j^H$ for $j \in \N$ are concordance invariants. 
\end{conjecture}
Conjecture \ref{conj:concordance homomorphisms} would imply that $\tau^H$ and $\phi_j^H$ are concordance homomorphisms. This is presently probably out of reach since none of the flavors of Khovanov-Rozansky homology are known to be functorial under link cobordisms, which is in turn complicated by the fact that there is no definition of Khovanov-Rozansky homology for arbitrary link diagrams.

\section{Algebraic properties}
We begin with a review of the most important algebraic properties of Khovanov-Rozansky homology. Subsection \ref{subsection:algebraic structure} describes the properties of interest as they presently appear in the literature. After a short detour to introduce a new set of grading conventions in Subsection \ref{subsection:grading conventions}, we restate the same properties in the way that will be used in this paper in Subsection \ref{subsection:our algebraic structure}. Throughout, let $\F$ be an arbitrary field with $\mathrm{char}(\F) = 0$ and $K \subset S^3$ a knot.

\subsection{Algebraic properties in the literature}\label{subsection:algebraic structure}
Let $\widehat{\mathit{H}}(K)$ denote the reduced Khovanov-Rozansky homology of $K$.
\begin{enumerate}
    \item \textbf{Gradings}: $\widehat{\mathit{H}}(K)$ is a finite dimensional $\Z\oplus\Z\oplus\Z$-graded vector space over $\F$ with grading $\gr_{\mathit{DGR}} = (\gr_a, \gr_q, \gr_t)$.
    \item \textbf{Symmetry}: There is an automorphism $i:\widehat{\mathit{H}}(K)\to\widehat{\mathit{H}}(K)$ which sends $(\gr_a, \gr_q, \gr_t)$ to $(\gr_a, -\gr_q, \gr_t-\gr_q)$.
    \item \textbf{Higher differentials}: For each $n \in \Z \setminus \{0\}$, there exists a vector space endomorphism $d_n: \widehat{\mathit{H}}(K)\to\widehat{\mathit{H}}(K)$ satisfying $d_n^2=0$. Its grading is
    $$
    \gr_{\mathit{DGR}}(d_n) = \begin{cases} (-2, 2n, -1) & | \ n>0 \\  (-2, 2n, -1+2n) & | \ n < 0. \end{cases}
    $$
    Moreover, $d_nd_m + d_md_n=0$ for all $m, n \in \N$.
    \item \textbf{$\mathfrak{sl}_2$ triple and higher operators}: There exists an $\mathfrak{sl}_2$ triple $(E, F, H)$ acting on $\widehat{H}(K)$ with gradings $\gr_{\mathit{DGR}}(E) = (0, -4, -2)$, $\gr_{\mathit{DGR}}(F) = (0, 4, 2)$, and $\gr_{\mathit{DGR}}(H) = (0, 0, 0)$. For each $(i, j) \in \N_0^2 \setminus \{(0, 0)\}$, the endomorphism $F$ gives rise to a higher operator $d_{i|j}: \widehat{\mathit{H}}(K)\to\widehat{\mathit{H}}(K)$ defined via $d_{i|j} = \frac{1}{j!}\ad_F^j(d_{i+j})$.
\end{enumerate}
This list is a result of deep work by many different authors. The gradings used above are the gradings that were predicted by \cite{dunfield2006superpolynomial} before the theory satisfying (1) was constructed in \cite{khovanov2004matrix}. The symmetry from point (2) was proven much later in \cite{oblomkov2020soergel} after a long series of papers that related it to algebraic geometry, representation theory and combinatorics. Higher differentials for $n>0$ were constructed in \cite[Section 5]{rasmussen2016some} and shown to have the grading as in point (3). The higher differentials for $n<0$ follow automatically from symmetry. The action of the $\mathfrak{sl}_2$ triple $(E, F, H)$ from point (4) was constructed in \cite{gorsky2024tautological} to provide an alternative proof of symmetry and based on that, the higher operators $d_{i|j}$ were considered in \cite{chandler2024structures}. We will use their results as stated in the following theorem.
\begin{theorem}
    \cite{gorsky2024tautological,chandler2024structures,sano2026ification}
    \label{ex-hypothesis}
    There exists an $\mathfrak{sl}_2$ triple $(E, F, H)$ acting on reduced Khovanov-Rozansky homology $\widehat{H}(K)$. In terms of grading $\gr_{\mathit{DGR}} = (\gr_a, \gr_q, \gr_t)$ the maps have degrees $\gr_{\mathit{DGR}}(E) = (0, -4, -2)$, $\gr_{\mathit{DGR}}(F) = (0, 4, 2)$ and $\gr_{\mathit{DGR}}(H) = (0, 0, 0)$. Moreover, the endomorphisms $E$, $F$, and $H : \widehat{H}(K) \to \widehat{H}(K)$ commute with the same chain homotopy equivalences that establish the invariance of Khovanov-Rozansky homology $\widehat{H}(K)$ under Reidemeister moves from \cite[Section 5.4]{rasmussen2016some}. Finally, there is a way to define the map $F$ on $\widehat{H}(K_1)\otimes\widehat{H}(K_2)$ such that for all $n, k \in \N$, the $k$-fold iterated commutator $\ad_F^k(d_n)$, and thus $d_{i|j} = \frac{1}{j!}\ad_F^k(d_n)$ commute with the chain homotopy equivalences that establish the connect sum formula in Khovanov-Rozansky homology \cite[Lemma 7.8]{rasmussen2016some}.
\end{theorem}
\begin{remark}
    The construction and properties of the $\mathfrak{sl}_2$ triple $(E, F, H)$ and the operators $d_{i|j}$ considered in \cite{gorsky2024tautological,chandler2024structures} are based on the map $F_2$. However, there is an issue with the proof of Theorem 6.4 (b) in \cite{gorsky2024tautological}, which constructs it; the formula defining $F_2$ in Equation (33) gives a well-defined endomorphism of $\mathcal{Y}(\beta)$, but it does not preserve the relevant ideal $I$ and hence does not descend to an endomorphism of $\mathcal{Y}(\beta)/I$. After discussing this point with the authors of \cite{gorsky2024tautological}, it was suggested that the gap in the proof could be rectified by using the map $\mathsf{e}$ defined by Sano \cite[Proposition 2.17]{sano2026ification} in place of $F_2$. Sano's map is well-defined and closely related, so the proofs of all of the results quoted in Theorem \ref{ex-hypothesis} that were originally proven in \cite{gorsky2024tautological,chandler2024structures} presumably remain valid when the correct $F_2$ is used. Nonetheless, I have not checked all of them myself and as far as I am aware, a complete presentation of the topic with the correct $F_2$ does not appear anywhere in the literature.
\end{remark}

\subsection{Grading conventions}\label{subsection:grading conventions}
Whenever knot homology theories are studied, there is always a lot of grading information. Different authors adopt different conventions and the sheer amount of seemingly random indices one is not intimately familiar with can sometimes obfuscate the mathematical ideas.

Somewhat ironically, this dissertation introduces a yet another set of gradings $\gr = (\gr_U, \gr_V, \gr_A)$. They are manifestly symmetric and we hope that $(\gr_a, \gr_q, \gr_t)$ and other grading conventions will gradually phase out in favor of $(\gr_U, \gr_V, \gr_A)$. A similar phenomenon has occurred in knot Floer homology where the literature has standardized with Maslov and Alexander gradings replaced by $(\gr_U, \gr_V)$. For transferring between $\gr$ and the grading $(\gr_a, \gr_q, \gr_t)$ from \cite{dunfield2006superpolynomial}, one can use the formulae
\begin{align*}
\gr_U &= \gr_a-\gr_t + \gr_q&    \gr_a &=\gr_A\\
\gr_V &= \gr_a - \gr_t \quad \quad \quad \quad \text{ or equivalently } &\gr_q &=\gr_U-\gr_V\\ 
\gr_A &= \gr_a&   \gr_t&= \gr_A-\gr_V.
\end{align*}
Another restatement of the grading relationship comes from the Poincar\'e polynomials of the Khovanov-Rozansky homology with respect to different gradings. If the Poincar\'e polynomial is given in terms of the variables $(a, q, t)$, one can substitute $a=UVA$, $q=U$, $t=U^{-1}V^{-1}$ to obtain the Poincar\'e polynomial in terms of the variables $(U,V,A)$. Conversely, the substitution $A=at, U=q, V=t^{-1}q^{-1}$ can be used to transform the Poincar\'e polynomial in terms of $(U,V,A)$ into one in terms of $(a,q,t)$.

\subsection{Algebraic properties with the new grading convention}\label{subsection:our algebraic structure}
We now present a condensed restatement of the structural properties of Khovanov-Rozansky homology from Subsection \ref{subsection:algebraic structure} in terms of the new gradings.
\begin{enumerate}
    \item \textbf{Gradings}: $\widehat{\mathit{H}}(K)$ is a finitely dimensional $\Z\oplus\Z\oplus\Z$-graded vector space over $\F$ with grading $\gr = (\gr_U, \gr_V, \gr_A)$.
    \item \textbf{Symmetry}: There is an automorphism $i:\widehat{\mathit{H}}(K)\to\widehat{\mathit{H}}(K)$ which sends $(\gr_U, \gr_V, \gr_A)$ to $(\gr_V, \gr_U, \gr_A)$.
    \item \textbf{Higher differentials and operators}: For each $(i,j) \in \N_0^2 \setminus\{(0,0)\}$, there exists an endomorphism $d_{i|j}: \widehat{\mathit{H}}(K)\to\widehat{\mathit{H}}(K)$. Its grading is
    $$
    \gr(d_{i|j}) = (2i-1, 2j-1, -2).
    $$
    Moreover, $d_{n|0}d_{m|0}+d_{m|0}d_{n|0}=0$ for all $m,n \in \N$.

    Additionally, there is an operator $F: \widehat{H}(K) \to \widehat{H}(K)$ satisfying Theorem \ref{ex-hypothesis} such that for all $(i, j) \in \N_0^2 \setminus \{(0,0)\}$ we have $d_{i|j} = \frac{1}{j}[F, d_{i+1|j-1}]$. Throughout the rest of the paper, we sometimes use $d_n$ to mean $d_{n|0}$.
\end{enumerate}

\section{The full version}\label{sec:definition}
Let the ring $\F[U,V]$ be equipped with a $\Z \oplus \Z \oplus \Z$ grading $\gr=(\gr_U, \gr_V, \gr_A)$ satisfying $\gr(U) = (-2, 0, 0)$ and $\gr(V)=(0, -2, 0)$. Let $\CH(K) = \widehat{\mathit{H}}(K) \otimes_{\F} \F[U,V]$ be the finitely generated free $\Z\oplus\Z\oplus\Z$-graded module over $\F[U,V]$. We can define a $\F[U,V]$-module endomorphism $\partial: \CH(K)\to\CH(K)$ via
$$\partial = \sum_{n=1}^\infty \ \sum_{i+j=n} U^iV^j d_{i|j}.$$
Recalling that $\gr(d_{i|j}) = (2i-1, 2j-1, -2)$, it follows that $\partial$ is homogeneous of degree $\gr(\partial) = (-1, -1, -2)$.

\subsection{Pictorial representation}\label{subsection:pictorial representation}
There exist various visual representations of the triple grading on $\widehat{\mathit{H}}(K)$, most notably the `flying saucer' diagrams from \cite{dunfield2006superpolynomial} and the diagrams sliced by the $\Delta$ grading from \cite{chandler2024structures}. Having introduced a new set of gradings, we also define a new pictorial representation of Khovanov-Rozansky homology that is best suited for these gradings. The idea is exactly the same as in the pictorial representation of knot Floer homology that was used in \cite{popovic2023link}.

The new pictures will mostly be used to depict the full version of the Khovanov-Rozansky chain complex $\CH(K)$, although the homology $\widehat{\mathit{H}}(K)$ can be recovered from the pictures as well by forgetting some of their features. The generators of $\CH(K)$ over $\F$ are drawn on a $2$-dimensional lattice where the $x$-coordinate denotes the $U^{-1}$ power and the $y$-coordinate denotes the $V^{-1}$ power. The arrows connecting the generators are decorated with the elements of $\F$ and indicate the differential $\partial$. For example, the horizontal arrow between $a$ and $b$ in Figure \ref{fig:finite CH} indicates that $\partial a = Ub$ and a vertical arrow between $c$ and $b$ in Figure \ref{fig:finite CH} indicates that $\partial c = Vb$. More generally, an arrow from $x$ to $y$ that moves by $i$ units in the horizontal direction, by $j$ units in the vertical direction and is decorated with $\lambda \in \F$ indicates that $\partial x$ contains the term $\lambda U^iV^j y$. An absence of a decoration means the same as a decoration with $1\in\F$.

Since $\gr(\partial)=(-1, -1, -2)$, the gradings $\gr_U-\gr_V$ and $\gr_U-\frac{1}{2}\gr_A$ are preserved by $\partial$ and one can calculate their values on $x$ and $U^iV^jx$ for some $x\in \CH(K)$ and $i, j \in \Z$. This gives that $(\gr_U-\gr_V)(U^iV^jx)=-2i+2j + (\gr_U-\gr_V)(x)$ and $(\gr_U-\frac{1}{2}\gr_A)(U^iV^jx) = -2i + (\gr_U-\frac{1}{2}\gr_A)(x)$. For any $(i, j) \neq (0,0)$, at least one of the gradings differ and therefore $x$ and $U^iV^jx$ are not connected via $\partial$. Formally speaking, the full Khovanov-Rozansky complex $\CH(K)$ splits as a chain complex over $\F$ (but \emph{not} as a chain complex over $\F[U,V]$) as a direct sum of $\Z^2$ many summands, one for each pair of values taken by the gradings ($\gr_U-\gr_V, \gr_U-\frac{1}{2}\gr_A$). Since the pictures corresponding to different summands are all translates of each other, we always only draw one of them and are intentionally ambiguous about the origin of the plane. 

The full Khovanov-Rozansky complexes $\CH(T_{2,-3})$ and $\CH(4_1)$ of the negative trefoil and the figure-eight knot are drawn in Figure \ref{fig:finite CH}. The complex $\CH(T_{2,-3})$ is freely generated by three generators $a$, $b$, and $c$. The differential satisfies $\partial a = Ub$, $\partial c = Vb$, and $\partial b = 0$. In general, the pictures will also feature some diagonal arrows corresponding to nonzero operators $d_{i|j}$ with $i,j > 0$.
\input{images/finite_CH.tex}

In the next two sections, we prove that $\partial^2=0$ and that the isomorphism class of $(\CH(K), \partial)$ is an invariant of $K$.

\section{\texorpdfstring{$\partial^2=0$}{d2=0}}\label{sec:d2=0}
This section contains the proof of the claim that $\partial$ is a differential, \emph{i.e.} that $\partial^2=0$. We do not find the proof to be particularly insightful -- rather, it consists of a long series of very technical algebraic manipulations. To shed some insight on the strategy of the underlying proof, we begin with an example. Motivated by some of its features, we state and prove Lemma \ref{lemma:d_{i|j}} and Lemma \ref{lemma:induction equation}. The section is concluded with the full proof that $\partial$ is a differential.

Note that 
$$\partial^2 = \sum_{n=1}^\infty \ \sum_{n_i+n_j=n} \ U^{n_i}V^{n_j}\sum_{\substack{i_1+i_2=n_i \\ j_1+j_2=n_j}}  \ d_{i_2|j_2}  d_{i_1|j_1}$$
where $n_i, n_j \in \N$ and $i_1, i_2, j_1, j_2 \in \N_0$ are such that none of the operators are $d_{0|0}$. To show that $\partial^2=0$, we will show in the proof of Theorem \ref{thm:d2=0} that each of the inner sums equals $0$, \emph{i.e.} that
$$\sum_{\substack{i_1+i_2=n_i \\ j_1+j_2=n_j}}  \ d_{i_2|j_2}  d_{i_1|j_1}=0$$
for all $n_i, n_j \in \N$.
\begin{example}\label{example:22}
Consider the case $n_i = n_j =2$, \emph{i.e.} the equality $$d_{1|2}d_{1|0} + d_{2|1}d_{0|1} +d_{0|2}d_{2|0} +d_{1|1}d_{1|1} +d_{2|0}d_{0|2} +d_{0|1}d_{2|1} +d_{1|0}d_{1|2} = 0$$
The geometric interpretation of the sum can be found in Figure \ref{fig:22}. It depicts two generators $x, y \in \widehat{H}(K)$ such that $y$ is $n_j=2$ units below and $n_i=2$ units to the left of $x$. The seven summands correspond to seven possible stopping points when one travels from $x$ to $y$ in two steps. For example, the stopping point $a$ corresponds to the summand $d_{1|2}d_{1|0}$.
\input{images/22.tex}

Let us partition the sum further based on the value of $m_1 = i_1+j_1$. For example, the case $m_1 = 2$ corresponds to the terms $d_{0|2}d_{2|0} +d_{1|1}d_{1|1} +d_{2|0}d_{0|2}$ on the long diagonal in Figure \ref{fig:22}. Reminded that we use the notation $d_2 = d_{2|0}$, we now calculate
\begin{align*}
d_{0|2}d_{2|0}& +d_{1|1}d_{1|1} +d_{2|0}d_{0|2}\\
&=\frac{1}{2}[F, d_{1|1}]d_{2} + [F, d_{2}][F, d_{2}] + d_{2} \frac{1}{2}[F, d_{1|1}]\\
&=\frac{1}{2}[F, [F, d_{2}]] d_{2} + [F, d_{2}][F, d_{2}] + d_{2}\frac{1}{2}[F, [F, d_{2}]]\\
&=\frac{1}{2}d_2F^2d_2-d_2Fd_2F+\frac{1}{2}d_2^2F^2 + Fd_2Fd_2-Fd_2^2F-d_2F^2d_2+d_2Fd_2F+\\
&\qquad +\frac{1}{2}F^2d_2^2-Fd_2Fd_2+\frac{1}{2}d_2F^2d_2\\
&=\frac{1}{2}d_2^2F^2 - Fd_2^2F + \frac{1}{2}F^2d_2^2\\
&=0
\end{align*}
because $d_2$ is a differential and satisfies $d_2^2=0$. What remains is to show that the sum of the remaining four terms with $m_1=1$ and $m_1=3$ also evaluates to $0$. We can similarly calculate
\begin{align*}
d_{1|0}d_{1|2}& +d_{0|1}d_{2|1} +d_{2|1}d_{0|1}+ d_{1|2}d_{1|0}\\
&= d_1 \frac{1}{2}[F,[F, d_3]] + [F, d_1][F, d_3] + [F, d_3][F, d_1] + \frac{1}{2}[F,[F, d_3]]d_1\\
&=\frac{1}{2} d_1F^2d_3 - d_1Fd_3F + \frac{1}{2}d_1d_3F^2+Fd_3Fd_1-Fd_3d_1F-d_3F^2d_1+\\
&\qquad +d_3Fd_1F+ Fd_1Fd_3 - Fd_1d_3F - d_1F^2d_3 + d_1Fd_3F + \frac{1}{2}F^2d_3d_1-\\
&\qquad -Fd_3Fd_1+\frac{1}{2}d_3F^2d_1\\
&=-\frac{1}{2}d_1F^2d_3 + \frac{1}{2}d_1d_3F^2 - F(d_3d_1+d_1d_3)F -\frac{1}{2}d_3F^2d_1 + d_3Fd_1F +\\
&\qquad +Fd_1Fd_3 + \frac{1}{2}F^2d_3d_1.
\end{align*}
Note that $d_1d_3+d_3d_1=0$ because $d_1$ and $d_3$ anti-commute. We are left with six terms. We now use that $[F,[F, d_1]]=0$, \emph{i.e.} that $F^2d_1-2Fd_1F+d_1F^2=0$. Precomposing or post-composing with $\frac{1}{2}d_3$ gives us two equations. Adding them together gives the result. Therefore, $\sum_{\substack{i_1+i_2=2 \\ j_1+j_2=2}}  \ d_{i_2|j_2}  d_{i_1|j_1}=0$ as required. 
\end{example}

Let us now treat the general case. In Example \ref{example:22}, our strategy was to replace the operators $d_{i|j}$ with compositions of $F$ and $d_{i+j}$ and identify cancellations. One can only do this manually for low values of $j$. In general, we have the following formula.
\begin{lemma}\label{lemma:d_{i|j}}
Let $i, j \in \N_0$ be such that they are not both $0$. Then
$$d_{i|j} = \frac{1}{j!}\sum_{k=0}^{j} (-1)^k \binom{j}{k} F^{j-k}d_{i+j}F^k.$$
\end{lemma}
\begin{proof}
By definition $d_{i|j} = \frac{1}{j!}\ad_F^j(d_{i+j})$ where $\ad_F^j(d_{i+j}) = [F, \dots, [F, d_{i+j}] \dots ]$ is the $j$-fold iterated adjoint operator. Expanding the commutator expression in full yields
$$\binom{j}{0}\underbrace{\addstackgap[10pt] F\cdots F}_{j}d_{i+j} - \binom{j}{1}\underbrace{\addstackgap[10pt ]F \cdots F}_{j-1}d_{i+j}F + \dots + (-1)^{j}\binom{j}{j} d_{i+j}F\cdots F$$
by a simple counting argument. Concretely, each of the $2^j$ terms in the expanded expression is of the form $F \cdots F d_{i+j} F \cdots F$ with different number of $F$s on the left and right. The signed number of terms which are precomposed by $F^k$ is $(-1)^k\binom{j}{k}$. The expression can be condensed into the formula $d_{i|j} = \frac{1}{j!}\sum_{k=0}^{j} (-1)^k \binom{j}{k} F^{j-k}d_{i+j}F^k$ as required.
\end{proof}
Another feature of Example \ref{example:22} was that the case of `shorter' diagonals was much more difficult than the case of `longer' diagonals. This was not a coincidence and is the case in general too. Recall that the last step in our calculation was using the formula $-d_1F^2 = F^2d_1-2Fd_1F$. This is a special case of the following equality with $m_1=1$ and $k=1$.
\begin{lemma}\label{lemma:induction equation}
Let $m_1, k \in \N$. Then
$$(-1)^{m_1}d_{m_1}F^{m_1+k} = \frac{(m_1+k)!}{(k-1)!}\ \sum_{c=0}^{m_1}\frac{(-1)^c}{(m_1-c)! \ c!} \cdot \frac{1}{m_1+k-c}F^{m_1+k-c}d_{m_1}F^c.$$
\end{lemma}
\begin{proof}
We prove this equality by induction on $k$. To prove the base case $k =1$, we use the relation $\ad_F^{m_1+1}(d_{m_1})= d_{-1|m_1+1}=0$. In its expanded form it is equivalent to
$$\sum_{c=0}^{m_1+1} (-1)^c \binom{m_1+1}{c} F^{m_1+1-c}d_{m_1}F^c =0.$$
Isolating the case $c=m_1+1$, this can be rewritten as
$$\sum_{c=0}^{m_1} (-1)^c \binom{m_1+1}{c} F^{m_1+1-c}d_{m_1}F^c =(-1)^{m_1}d_{m_1}F^{m_1+1},$$
which is exactly the equality obtained by substituting $k=1$ in the original equation. This establishes the base case.

Let us now assume that the original equation is true for some $k \in \N$. Then
$$(-1)^{m_1}d_{m_1}F^{m_1+k+1} = \frac{(m_1+k)!}{(k-1)!}\ \sum_{c=0}^{m_1}\frac{(-1)^c}{(m_1-c)! \ c!} \cdot \frac{1}{m_1+k-c}F^{m_1+k-c}d_{m_1}F^{c+1}$$
by the induction hypothesis. We now perform a series of algebraic manipulations on the right-hand side expression. By reindexing the sum, it is equal to
$$
\frac{(m_1+k)!}{(k-1)!}\ \sum_{c=1}^{m_1+1}\frac{(-1)^{c-1}}{(m_1+1-c)! \ (c-1)!} \cdot \frac{1}{m_1+k+1-c}F^{m_1+k+1-c}d_{m_1}F^{c}.
$$
Isolating the summand $c=m_1+1$ gives the expression
\begin{align*}
&\frac{(m_1+k)!}{(k-1)!}\ \sum_{c=1}^{m_1}\frac{(-1)^{c-1}}{(m_1+1-c)! \ (c-1)!} \cdot \frac{1}{m_1+k+1-c}F^{m_1+k+1-c}d_{m_1}F^{c}\\
&\qquad + \frac{(m_1+k)!}{(k-1)!}\cdot \frac{(-1)^{m_1}}{0! \ m_1!} \cdot \frac{1}{k} F^kd_{m_1}F^{m_1+1}.
\end{align*}
We notice that $(-1)^{m_1} d_{m_1}F^{m_1+1}$ is hidden in plain sight within the last term of the expression. This lets us use the base case to rewrite the expression yet again as
\begin{align*}
&\frac{(m_1+k)!}{(k-1)!}\ \sum_{c=1}^{m_1} \frac{(-1)^{c-1}}{(m_1+1-c)! \ (c-1)!} \cdot \frac{1}{m_1+k+1-c} F^{m_1+k+1-c}d_{m_1}F^{c}\\
&\qquad + \frac{(m_1+k)!}{k! \ m_1!}\ F^k \sum_{c=0}^{m_1}(-1)^c \binom{m_1+1}{c} F^{m_1+1-c}d_{m_1}F^c.
\end{align*}
We now isolate the summand $c=0$ from the second sum and then combine both sums into one. This gives us
\begin{align*}
&\frac{(m_1+k)!}{k!}\ \sum_{c=1}^{m_1} \left( \frac{k (-1)^{c-1}}{(m_1+1-c)! (c-1)!} \cdot \frac{1}{m_1+k+1-c} + \frac{(m_1+1) (-1)^c}{c! (m_1+1-c)!} \right)
F^{m_1+k+1-c}d_{m_1}F^{c}\\
&\qquad + \frac{(m_1+k)!}{k! \ m_1!} F^{m_1+k+1}d_{m_1}F^0,
\end{align*}
and after some rearranging
\begin{align*}
&\frac{(m_1+k)!}{k!}\ \sum_{c=1}^{m_1} \frac{(-1)^c}{(m_1-c)! \ c!}\left( \frac{-kc}{m_1+1-c} \cdot \frac{1}{m_1+k+1-c} + \frac{m_1+1}{m_1+1-c} \right)
F^{m_1+k+1-c}d_{m_1}F^{c}\\
&\qquad + \frac{(m_1+k+1)!}{k!} \cdot \frac{1}{m_1!\ 0!}\cdot \frac{1}{m_1+k+1} F^{m_1+k+1}d_{m_1}F^0.
\end{align*}
The expression inside the parentheses evaluates to $\frac{m_1+k+1}{m_1+k+1-c}$. This means that the entire expression is equal to
$$\frac{(m_1+k+1)!}{k!}\ \sum_{c=0}^{m_1} \frac{(-1)^c}{(m_1-c)! \ c!}\cdot \frac{1}{m_1+k+1-c}
F^{m_1+k+1-c}d_{m_1}F^{c},$$
which is precisely the right-hand side of the original equation for $k+1$. This concludes the inductive step and the proof of the lemma.
\end{proof}
We are now able to conclude that $\partial$ is a differential.
\begin{theorem}\label{thm:d2=0}
$\partial^2=0$.
\end{theorem}
\begin{proof}
To show that $\partial^2=0$, we will show that $$\sum_{\substack{i_1+i_2=n_i \\ j_1+j_2=n_j}}  \ d_{i_2|j_2}  d_{i_1|j_1}=0$$
for all $n_i, n_j \in \N$. Without loss of generality, we assume that $n_j \leq n_i$. The other case follows by the symmetry of Khovanov-Rozansky homology or an analogous computation.

Armed with the formula $d_{i|j} = \frac{1}{j!}\sum_{k=0}^{j} (-1)^k \binom{j}{k} F^{j-k}d_{i+j}F^k$ from Lemma \ref{lemma:d_{i|j}}, we can inspect the individual summands from $\sum_{\substack{i_1+i_2=n_i \\ j_1+j_2=n_j}}  \ d_{i_2|j_2}  d_{i_1|j_1}$. Expanding the commutators we observe that each term is of the form $F^a d_{m_2} F^b d_{m_1} F^c$ for some $a, b, c \in \N_0$ with $a+b+c=n_j$ and $m_1, m_2 \in \N$ with $m_1+m_2 = n$. In what follows we determine the coefficient of each term and show that those with nonzero coefficients cancel each other out.

Define the sets $A_1=\{ n_j, \dots, n_i \}$ and $A_2 = \{1, \dots, n_j-1, n_i+1, \dots, n-1\}$ and note that they have the property that if $k \in A_l$, then also $n-k \in A_l$ for $l \in \{1,2\}$. This means that both $m_1$ and $m_2$ either belong to $A_1$ or to $A_2$ and separates our analysis into two cases.
\begin{enumerate}
    \item Case 1: $m_1, m_2 \in A_1$.

    In particular, this means that $m_1, m_2 \geq n_j = a + b +c$. The inequalities we need are even slightly weaker, $m_1 \geq b+c$ and $m_2 \geq a+b$. Together they imply that within the term $F^ad_{m_2}F^bd_{m_1}F^c$, the part $F^ad_{m_2}F^{b-k}$ can come from $d_{i_2|j_2}$ and the part $F^{k}d_{m_1}F^c$ can come from $d_{i_1|j_1}$ for any $k \in \{0, \dots, b\}$. Fixing a $k$, the coefficient of $F^ad_{m_2}F^bd_{m_1}F^c$ is $$\frac{1}{(a+b-k)!} (-1)^{b-k}\binom{a+b-k}{b-k} \cdot \frac{1}{(c+k)!} (-1)^c \binom{c+k}{c},$$
    which is equal to $\frac{(-1)^{b+c-k}}{a! \ b! \ c!} \binom{b}{k}$ after simplification. Summing over all $k$ we obtain $\frac{(-1)^{b+c}}{a! \ b! \ c!} \sum_{k=0}^b (-1)^{k} \binom{b}{k}$. Provided $b \geq 1$, this evaluates to $0$ due to a well-known binomial identity, which means that the term $F^ad_{m_2} F^b d_{m_1} F^c$ does not appear in the expansion. If $b=0$, then $F^ad_{m_2}d_{m_1}F^c$ does appear, but so does $F^ad_{m_1}d_{m_2}F^c$ with the same coefficient. Since the differentials $d_{m_1}$ and $d_{m_2}$ anti-commute, \emph{i.e.}, $d_{m_2}d_{m_1}+d_{m_1}d_{m_2}=0$, the two terms cancel each other.

    \item Case 2: $m_1, m_2 \in A_2$.

    Let's first assume $m_1 < m_2$. This means that $m_1 < n_j = a+b+c$ and $m_2 > n_i \geq n_j = a+b+c \geq a+b$. As in case (1), we wish to partition $F^ad_{m_2}F^bd_{m_1}F^c$ into parts $F^ad_{m_2}F^{b-k}$ and $F^kd_{m_1}F^c$ that come from $d_{i_2|j_2}$ and $d_{i_1|j_1}$ respectively. However, not all values $k \in \{0, \dots, b\}$ are allowed anymore in general. We need $c+k \leq m_1$ or equivalently $k \leq m_1-c$. In particular, this means that $F^ad_{m_2}F^bd_{m_1}F^c$ appears with coefficient $0$ if $c > m_1$. In everything that follows we thus assume that $c \leq m_1$. We now have two cases.
    \begin{enumerate}
        \item $b = 0$.
        
        The coefficient of $F^ad_{m_2}d_{m_1}F^c$ is $(-1)^c\frac{1}{a! \ c!}$ as established in Case 1 above.
        \item $b \geq 1$.
        \begin{enumerate}
            \item $b+c \leq m_1$. Then the condition $k \leq m_1-c$ is vacuous and the sum $\frac{(-1)^{b+c}}{a! \ b! \ c!} \sum_{k=0}^b (-1)^{k} \binom{b}{k}$ evaluates to $0$ as in Case 1. Such terms $F^ad_{m_2}F^bd_{m_1}F^c$ therefore do not appear in the sum.
            \item $b+c = m_1 +k$ for some $k \in \N$. The condition $k \leq m_1-c$ means that the term $F^ad_{m_2}F^bd_{m_1}F^c$ has coefficient
            \begin{align*}
            &\frac{(-1)^{m_1+k}}{a! \ (m_1+k-c)! \ c!} \sum_{h=0}^{m_1-c} (-1)^{h} \binom{m_1+k-c}{h}\\
            &\qquad = \frac{(-1)^{m_1+k}(-1)^{m_1-c}}{a! \ (m_1+k-c)! \ c!} \binom{m_1+k-c-1}{m_1-c}\\
            &\qquad =\frac{(-1)^{k-c}}{a! \ } \cdot \frac{1}{(k-1)!} \cdot \frac{1}{(m_1-c)! \ c!} \cdot \frac{1}{m_1+k-c}.
            \end{align*}
            Summing them over all $c$ we obtain 
            $$\sum_{c=0}^{m_1} \frac{(-1)^k}{a!}\cdot\frac{1}{(k-1)!}\cdot \frac{(-1)^c}{(m_1-c)! \ c!} \cdot \frac{1}{m_1+k-c}F^{m_1+k-c}d_{m_1}F^c,$$
            which is equal to 
            $$(-1)^{m_1+k}\frac{1}{a! \ (m_1+k)!}F^ad_{m_2}d_{m_1}F^{m_1+k}$$
            by Lemma \ref{lemma:induction equation}.
        \end{enumerate}
    \end{enumerate}
    Let us now summarize the terms that are still present in the sum. From the case $b=0$, we obtained the terms $(-1)^{n_j-a}\frac{1}{a! \ (n_j-a)!}F^ad_{m_2}d_{m_1}F^{n_j-a}$. The case $b \geq 1$ gave us several terms with nonzero coefficients, but we condensed them into the terms $(-1)^{m_1+k} \frac{1}{a! \ (m_1+k)!} F^ad_{m_2}d_{m_1} F^{m_1+k}$. Since $b+c = m_1+k = n_j -a$, they actually only differ in their description. We can add them together into
    $$\sum_{h = 0}^{n_j}(-1)^h \frac{1}{(n_j-h)! \ h!} F^{n_j-h}d_{m_2}d_{m_1} F^h = \ad_F^{n_j}(d_{m_2}d_{m_1}).$$
    Recall that we specialized to the case $m_1 < m_2$ at the beginning of Case 2. Running the same argument with $m_1 > m_2$ gives another summand $\ad_F^{n_j}(d_{m_1}d_{m_2})$. The differentials $d_{m_1}$ and $d_{m_2}$ anti-commute, \emph{i.e.}, $d_{m_1}d_{m_2}+d_{m_2}d_{m_1}=0$ and $\ad_F$ is linear. Therefore $\ad_F^{n_j}(d_{m_2}d_{m_1}) + \ad_F^{n_j}(d_{m_1}d_{m_2})=0$ as required.
\end{enumerate}
\end{proof}

\section{Invariance and connected sum isomorphism}
The main result of this section is the proof that the isomorphism class of $(\CH(K), \partial)$ is a knot invariant, \emph{i.e.} remains unchanged under Reidemeister moves. This essentially amounts to giving precise citations to the invariance results of the hat version of Khovanov-Rozansky homology by Khovanov-Rozansky \cite{khovanov2008matrix}, the differentials $d_{n|0}$ by Rasmussen \cite{rasmussen2016some}, the $\mathfrak{sl}_2$ triple $(E, F, H)$ action from Theorem \ref{ex-hypothesis}, and justifying that all mentioned invariance results are compatible.
\begin{theorem}\label{thm:invariance}
$\partial$ is a knot invariant.
\end{theorem}
\begin{proof}
Let $K \subset S^3$ be a knot and let $D$ be a braid diagram of $K$. Khovanov and Rozansky constructed in \cite{khovanov2008matrix} a chain complex of matrix factorizations that is associated to $D$ and whose homology is (by definition) Khovanov-Rozansky homology. Using the language of double chain complexes $(C(D), d_{+}, d_v)$, Rasmussen discovered an original reformulation of their construction in \cite{rasmussen2016some}. He also constructed an additional infinite family of higher differentials $d_n$ for $n \in \{-1\} \cup \N$ that give rise to spectral sequences $E(n)$ from $\widehat{\mathit{H}}(K)$ to $\mathfrak{sl}_n$ homology \cite[Theorem 2]{rasmussen2016some}. Following Khovanov and Rozansky's proof of invariance of $\widehat{\mathit{H}}(K)$, Rasmussen showed that all differentials $d_n$ commute with the \emph{same} chain homotopy equivalences \cite[Section 5.4]{rasmussen2016some}. Therefore, all $E(n)$ and in particular all $d_n$ are knot invariants. In the language of this paper $d_{n|0} = d_n$. Moreover, by Theorem \ref{ex-hypothesis}, the endomorphisms $E, F, H : \widehat{\mathit{H}}(K) \to \widehat{\mathit{H}}(K)$ commute with these chain homotopy equivalences as well. It follows in particular that $F$ is a knot invariant. Since the operators $d_{i|j}$ are compositions of $F$ and $d_{i+j}$ by Lemma \ref{lemma:d_{i|j}}, it follows that all $d_{i|j}$ are knot invariants and hence so is $\partial$.
\end{proof}

We can now prove the main result of this paper.
\begin{proof}[Proof of Theorem \ref{thm:full HOMFLY complex}]
$(\CH(K), \partial)$ is defined as a $\Z\oplus\Z\oplus\Z$-graded free $\F[U,V]$-module with an endomorphism in Section \ref{sec:definition}. Section \ref{sec:d2=0} culminating in Theorem \ref{thm:d2=0} shows that $\CH(K)$ is a chain complex and Theorem \ref{thm:invariance} shows that its isomorphism type is a knot invariant. This proves (1). For (2), note that setting $U=V=0$ simply discards the additional structure we sought to investigate in this paper and thus recovers the reduced Khovanov-Rozansky homology $\widehat{\mathit{H}}(K)$. 
\end{proof}

Our next goal is to establish that the connected sum of knots corresponds to a tensor product of Khovanov-Rozansky complexes, \emph{i.e.} we prove Theorem \ref{thm:connected sums}.
\begin{proof}[Proof of Theorem \ref{thm:connected sums}]
The corresponding result $\widehat{\mathit{H}}(K_1 \# K_2) \cong \widehat{\mathit{H}}(K_1) \otimes_{\F} \widehat{\mathit{H}}(K_2)$ for the reduced Khovanov-Rozansky homology was already proven in \cite[Lemma 7.8]{rasmussen2016some}. It follows that the two sides of the putative isomorphism $\CH(K_1\# K_2) \cong \CH(K_1) \otimes_{\F[U,V]} \CH(K_2)$ are isomorphic as $\Z\oplus\Z\oplus\Z$ graded free $\F[U,V]$-modules. It remains to upgrade this relation to an isomorphism of chain complexes over $\F[U,V]$ by studying $\partial$. In light of the definition of $\partial$ as $\partial = \sum_{n=1}\sum_{i+j=n} U^iV^j d_{i|j}$, it is sufficient to understand the endomorphisms $d_{i|j}$, which are in turn iterated compositions of $d_{n|0}$ and the endomorphism $F$ by Lemma \ref{lemma:d_{i|j}}. By Theorem \ref{ex-hypothesis} with $k=j$ and $n=i+j$, we have that $\ad_F^j(d_{i+j}) = j! \, d_{i | j}$ commutes with the isomorphism from \cite[Lemma 7.8]{rasmussen2016some}. It follows that so does $d_{i|j}$ and hence $\partial$. Therefore, $\CH(K_1 \# K_2) \cong \CH(K_1) \otimes_{\F[U,V]} \CH(K_2)$ as required.
\end{proof}

\section{Structure theorem and comparison with knot Floer homology}\label{sec:structure thm}
This section considers an intermediate version of Khovanov-Rozansky homology. Instead of working over the ring $\F$ (as in the classical setup of Khovanov-Rozansky homology) or over the ring $\F[U,V]$ (as in full Khovanov-Rozansky homology studied in this paper so far), we investigate the Khovanov-Rozansky chain complex $\CH_{\mathcal{R}_1}(K) = \CH(K) \otimes_{\F[U,V]} \mathcal{R}_1$ over $\mathcal{R}_1 = \frac{\F[U,V]}{(UV)}$. Our main motivation for this pursuit comes from knot Floer homology where the algebraic structure of such complexes has been leveraged to extract concordance information from $\CFK_{\mathcal{R}_1}(K)$. With the aim of replicating this development in the Khovanov-Rozansky setting, we begin by classifying the chain homotopy types of Khovanov-Rozansky chain complexes $\CH_{\mathcal{R}_1}(K)$ in Theorem \ref{thm:HOMFLY splitting}. We then comment on the results and compare them with knot Floer homology.

In order to state the structure theorem for Khovanov-Rozansky homology, we need the standard complexes, snake complexes and local systems. All of these types of chain complexes over $\mathcal{R}_1$ were already considered in the knot Floer homology context. Standard complexes were defined in \cite{dai2021more} and generalized into snake complexes in \cite{popovic2023link}, which is also where the local systems were introduced. 
\begin{definition}
\label{def:standard complex}
Let $n \in 2\N_0$ and let $a_1, \dots, a_n$ be a sequence of nonzero integers. The \emph{standard complex} $C(a_1, \dots, a_n)$ is a free chain complex over $\mathcal{R}_1$ with a distinguished basis $B=\{x_{0}, \dots, x_n\}$ and a differential $\partial$ defined as follows. For each odd $i$, there is a horizontal arrow of length $|a_i|$ connecting $x_i$ and $x_{i-1}$. For each even $i$, there is a vertical arrow of length $|a_i|$ connecting $x_i$ and $x_{i-1}$. The direction of the arrow is determined by the sign of $a_i$, as follows. If $a_i > 0$, then the arrow goes from $x_i$ to $x_{i-1}$, and if $a_i < 0$, then the arrow goes from $x_{i-1}$ to $x_i$. The $\Z\oplus\Z\oplus\Z$ grading on $C(a_1, \dots, a_n)$ is uniquely determined by the condition $\gr(x_0)=(0, 0, 0)$ and that $\gr(\partial) = (-1, -1, -2)$.
\end{definition}
\begin{definition}
\label{def:snake complex}
\cite[Definition 2.12]{popovic2023link}
Let $m \in 2\N_0+1$ and let $b_1, \dots, b_{m}$ be a sequence of nonzero integers. A \emph{horizontal snake complex} $S_h(b_1, \dots, b_{m})$ is a free chain complex over $\mathcal{R}_1$ with a distinguished basis $B=\{x_0, \dots, x_m\}$ and a differential $\partial$ defined as follows. For each odd $i$, there is a horizontal arrow of length $|b_i|$ connecting $x_i$ and $x_{i-1}$. For each even $i$, there is a vertical arrow of length $|b_i|$ connecting $x_i$ and $x_{i-1}$. The direction of the arrow is determined by the sign of $b_i$, as follows. If $b_i > 0$, then the arrow goes from $x_i$ to $x_{i-1}$, and if $b_i < 0$, then the arrow goes from $x_{i-1}$ to $x_i$. The $\Z\oplus\Z\oplus\Z$ grading on $S_h(b_1, \dots, b_m)$ is allowed to be arbitrary subject to the condition that $\gr(\partial) = (-1, -1, -2)$.

Similarly, let $m \in 2\N_0$ and let $b_0, \dots, b_m$ be a sequence of nonzero integers. A \emph{vertical snake complex} $S_v(b_0, \dots, b_m)$ is a free chain complex over $\mathcal{R}_1$ with a distinguished basis $B=\{x_{-1}, \dots, x_m\}$ and a differential $\partial$ defined exactly as above.
    
A \emph{snake complex} refers to either a standard complex, a horizontal snake complex, or a vertical snake complex.
\end{definition}
For internalizing the notions of different flavors of snake complexes, Figure \ref{fig:example complexes} is likely to be a lot more illuminating than the preceding definitions. Horizontal snake complexes start and end with horizontal arrows, vertical snake complexes start and end with vertical arrows and standard complexes start with a horizontal and end with a vertical arrow.
\input{images/example_of_a_standard_complex.tex}

\begin{definition}
Let $(L, \partial)$ be a finitely generated free chain complex over $\mathcal{R}_1$ with no arrows of length 0. Then $L$ is an \emph{indecomposable local system} if it
\begin{enumerate}
    \item admits a simplified decomposition (Definition 3.1 in \cite{popovic2023link}),
    \item is indecomposable as a chain complex over $\mathcal{R}$, and
    \item has torsion homology (Definition 2.2 in \cite{popovic2023link}).
\end{enumerate}
The $\Z\oplus\Z\oplus\Z$ grading on $S_h(b_1, \dots, b_m)$ is allowed to be arbitrary subject to the condition that $\gr(\partial) = (-1, -1, -2)$.

A \emph{local system} is a direct sum of indecomposable local systems of the same shape (Definition 3.5 in \cite{popovic2023link}) and in the same position in the plane. A local system is trivial if it admits a simplified basis (Definition 2.5 in \cite{popovic2023link}) and nontrivial otherwise.
\end{definition}

\vspace{1em}

Because the papers \cite{dai2021more,popovic2023link} study knot Floer homology, they defined standard complexes, snake complexes and local systems to be $\Z\oplus\Z$ graded with a differential $\partial$ of degree $\gr(\partial)=(-1,-1)$. The Khovanov-Rozansky chain complex $\CH_{\mathcal{R}_1}(K)$ admits an additional grading $\gr_A$, \emph{i.e.} it is $\Z\oplus\Z\oplus\Z$ graded and the differential $\partial$ has degree $\gr(\partial)=(-1,-1,-2)$. This small difference leads to some additional rigidity in this class of complexes which is explored further later in this section. For now, disregarding the third grading yields the main structure theorem.
\begin{proof}[Proof of Theorem \ref{thm:HOMFLY splitting}]
Since $\widehat{\mathit{H}}(K)$ is a $\Z\oplus\Z\oplus\Z$ graded finite dimensional vector space, $\CH_{\mathcal{R}_1}(K) = \widehat{\mathit{H}}(K) \otimes_{\F} \mathcal{R}_1$ is a finitely generated free $\Z\oplus\Z\oplus\Z$ graded chain complex over $\mathcal{R}_1$. Forgetting the third grading $\gr_A$, this is exactly the class of complexes studied in Theorem 4.2 of \cite{popovic2023link} and the conclusion follows.
\end{proof}
Although very similar, the structure of $\CH_{\mathcal{R}_1}(K)$ is not identical to the structure of $\CFK_{\mathcal{R}_1}(K)$. In particular, Theorem \ref{thm:HOMFLY splitting} allows an arbitrary number of snake complexes rather than a single standard complex. Nonetheless, in all examples that have been computed so far, which includes all knots with $\leq 11$ crossings \cite{nakagane2025computations, chandler2024structures}, $\CH_{\mathcal{R}_1}(K)$ contains a unique standard complex and we conjecture that this is true in general. Proving this claim would likely come from an isomorphism $\frac{H_*(\CH_{\mathcal{R}_1}(K)/U)}{V-\text{torsion}} \cong \F[V]$, something that is not established very easily since it would require understanding $\partial$, and with it the interactions between different $d_i$'s, better.

Let us also remark that experimental evidence from \cite{nakagane2025computations, chandler2024structures} suggests that $\CH_{\mathcal{R}_1}(K)$ always splits into a standard complex of the form $C(1, -1, \dots, 1, -1)$ or $C(-1, 1, \dots, -1, 1)$ together with some local systems that are $1 \times 1$ squares. More precisely, it is observed in \cite[Proposition 3.31]{chandler2024structures} that this happens for all knots with $\leq 11$ crossings and more generally whenever Rasmussen's $\mathfrak{sl}_1$ spectral sequence collapses after the first differential $d_1$. We do \emph{not} expect or hope this to be true generically, but nonetheless note that Theorem \ref{thm:HOMFLY splitting} can be used to provide a different proof of this result that does not directly appeal to the $\mathfrak{sl}_2$-action.
\begin{proposition}\label{prop:HOMFLY width 1}
Let $K \subset S^3$ be a knot and $\widehat{\mathit{H}}(K)$ its reduced Khovanov-Rozansky homology. If the Rasmussen's $\mathfrak{sl}_1$ spectral sequence from $\widehat{\mathit{H}}(K)$ to $\F$ collapses after the first differential $d_1$, then $\CH(K)$ splits as a direct sum of
\begin{itemize}[label=$-$]
    \item $C(1, -1, \dots, 1, -1)$ or $C(-1, 1, \dots, -1, 1)$, and
    \item some $1 \times 1$ squares.
\end{itemize}
One of the sides of each $1 \times 1$ square is decorated with a $-1$ and all other arrows are decorated with a $1$.
\end{proposition}
\begin{proof}
We first determine the corresponding result for $\CH_{\mathcal{R}_1}(K)$ and then lift it to $\CH(K)$ as in the statement of the theorem. Let us choose a basis for $\CH_{\mathcal{R}_1}(K)$ in which the complex splits into snake complexes and local systems as in Theorem \ref{thm:HOMFLY splitting}. The collapse of the $\mathfrak{sl}_1$ spectral sequence after $d_1$ can be rephrased by saying that $H_*(\widehat{\mathit{H}}(K), d_{1|0}) \cong \F$. This means that all but one of the generators of $\CH(K)$ are adjacent to a horizontal arrow of length $1$ in our pictorial representation of $\CH(K)$. Since each generator of $\CH_{\mathcal{R}_1}(K)$ is adjacent to horizontal arrows of a unique length, this means that the picture only contains horizontal arrows of length $1$. By symmetry, we also have that $H_*(\widehat{\mathit{H}}(K), d_{0|1}) \cong \F$ and every generator except one is adjacent to a vertical arrow of length $1$, so all arrows in the picture have length $1$. If
$$
    
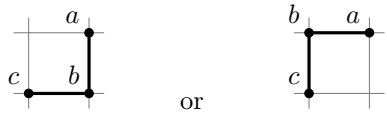
\begin{tikzpicture}[scale = 0.8]
        \draw[step=1.0,gray,thin] (-0.25, -0.25) grid (1.25,1.25);
        \draw[black, very thick] (0,0) -- (1,0) -- (1,1);
        \filldraw[black] (1,1) circle (2pt) node[anchor=south east]{$a$};
        \filldraw[black] (0,0) circle (2pt) node[anchor=south east]{$c$};
        \filldraw[black] (1,0) circle (2pt) node[anchor=south east]{$b$};
    \end{tikzpicture}
    \hspace{1cm}
    \text{or}
    \hspace{1cm}
    \begin{tikzpicture}[scale = 0.8]
        \draw[step=1.0,gray,thin] (-0.25, -0.25) grid (1.25,1.25);
        \draw[black, very thick] (0,0) -- (0,1) -- (1,1);
        \filldraw[black] (1,1) circle (2pt) node[anchor=south east]{$a$};
        \filldraw[black] (0,0) circle (2pt) node[anchor=south east]{$c$};
        \filldraw[black] (0,1) circle (2pt) node[anchor=south east]{$b$};
    \end{tikzpicture}
$$
ever appear in the picture with arbitrary decorations, then $\partial^2 a = 0$ requires that there be at least one other way of travelling from $a$ to $c$ in two steps. Since all arrows are pointing to the left and downwards, there is exactly one way how this can happen -- the new way must pass through a generator $d$ in the last corner of the square. Simple changes of bases can now be used to guarantee that three of the square sides are decorated with $1$ and one of them is decorated with a $-1$. The only valid arrangement of arrows of length $1$ that avoids both patterns depicted above is an alternating sequence of arrows of length $1$ in the northwest-southeast direction. Since exactly one of the generators is not adjacent to any horizontal arrows, this must be a standard complex, \emph{i.e.} either $C(1, -1, \dots, 1, -1)$ or $C(-1, 1, \dots, -1, 1)$. A change of basis can be used to guarantee that all arrows are decorated with a $1$, concluding the proof.
\end{proof}
See also \cite[Lemma 7]{petkova2013cables} for a closely related and much older result phrased in the language of knot Floer homology. It has recently been generalized in \cite{popovic2025knots}, which suggests some local systems that might appear in Khovanov-Rozansky homology if the Rasmussen's spectral sequence fails to collapse.

There is one more difference between knot Floer and Khovanov-Rozansky homology -- the presence of the third grading $\gr_A$ enables us to say something more about $\CH_{\mathcal{R}_1}(K)$.
\begin{proposition}\label{prop:no pathological behavior}
$\CH_{\mathcal{R}_1}(K) \cong C \otimes_{\F} \mathcal{R}_1$ for some finitely generated chain complex $C$ over $\F$.
\end{proposition}
\begin{proof}
This was already proven in Section \ref{subsection:pictorial representation} where the pictorial representations of $\CH(K)$ were defined. The proposition is just a formal restatement of the fact that the picture of $\CH(K)$ splits into $\Z^2$ many disconnected summands, each containing one of the elements $U^iV^jx$ for some $x \in \CH(K)$.
\end{proof}
This is in contrast with the knot Floer setting, where \emph{essentially infinite} chain complexes can exist. See Examples $T$ and $D$ in \cite{popovic2023link} for complexes that cannot be drawn in the plane as infinitely many copies of a finite picture. The intuitive content of Proposition \ref{prop:no pathological behavior} is that the third grading $\gr_A$ in Khovanov-Rozansky homology rules out the possibility of such pathological behavior.

Finally, although it can be defined for an arbitrary field, knot Floer homology is typically studied over $\F = \Z/2\Z$. At the moment, this field cannot be used for full Khovanov-Rozansky homology since the paper \cite{rasmussen2016some} requires a field $\F$ with $\mathrm{char}(\F) = 0$ for the construction of differentials. The most studied versions of the two knot homologies are thus not directly comparable, but it would be very interesting nonetheless to investigate the relationship between the two decompositions into snake complexes and local systems. 

\section{New invariants}\label{sec:new invariants}
In this section, the algebraic structure of the Khovanov-Rozansky chain complex $\CH_{\mathcal{R}_1}(K)$ is used to define some knot invariants: $\tau^H$, $\epsilon^H$, and $\phi_j^H$ for all $j \in \N$. We prove they are additive under connected sum of knots, but do not know whether they vanish on all slice knots.

The motivation for $\tau^H$, $\epsilon^H$ and $\phi^H_j$ once again comes from knot Floer homology, although it is not immediately clear how to define analogues of $\tau$, $\epsilon$, and $\phi_j$ since $\CH_{\mathcal{R}_1}(K)$ can potentially contain multiple snake complexes. We begin with the definition of the following relation.
\begin{definition}
    Let $C_1$ and $C_2$ be finitely generated free chain complexes over $\mathcal{R}_1$.
    Define $C_1 \sim C_2$ if $C_1 \oplus L_1 \cong C_2 \oplus L_2$ for some direct sums of local systems $L_1$, $L_2$.
\end{definition}
The direct sums of local systems $L_1$ and $L_2$ in the above definition are allowed to be empty. It is immediate to verify that $\sim$ is an equivalence relation and one should think of its equivalence classes as complexes ``up to local systems''.

We now investigate how $\sim$ behaves under the tensor product of chain complexes over $\mathcal{R}_1$. In the statement of the following lemma, let each instance of $S_h$, $S_v$, $C$, and $L$ denote some, not necessarily the same in all occurrences, horizontal snake complex, vertical snake complex, standard complex, and local system respectively.
\begin{lemma}\label{lemma:sim relations}
We have the following relations up to $\sim$:
\[
    \begin{tabular}{c|ccccccccc}
    $\otimes_{\mathcal{R}_1}$   & $C$   & $S_h$   & $S_v$ & $L$ \\
    \hline\vrule height 12pt width 0pt
    $C$   & $C$   & $S_h$   & $S_v$ & $0$  \\
    $S_h$   & $S_h$   & $S_h \oplus S_h$ & $0$   & $0$  \\
    $S_v$ & $S_v$ & $0$   & $S_v \oplus S_v$   & $0$ \\
    $L$   & $0$   & $0$   & $0$   & $0$ \\
    \end{tabular} 
\]
In words, a tensor product of two standard complexes is $\sim$ equivalent to a standard complex, a product of a standard complex and a horizontal snake complex is $\sim$ equivalent to a horizontal snake complex, and so on. 
\end{lemma}
\begin{proof}
Standard complexes, horizontal snake complexes, vertical snake complexes, and local systems are different types of finitely generated free chain complexes over $\mathcal{R}_1$ and all of those properties are preserved under tensor product. Let $X$ and $Y$ be two such complexes. By the main structural result \cite[Theorem 4.2]{popovic2023link}, we thus know that $X \otimes_{\mathcal{R}_1} Y$ splits as a direct sum of complexes of the aforementioned types. Since we are only interested in $X \otimes_{\mathcal{R}_1}Y$ up to $\sim$, it is sufficient to determine the numbers of snake complexes of each type. This is done by examining the ranks of $\frac{H_*((X \otimes_{\mathcal{R}_1}Y) / U)}{V\text{-torsion}}$ and $\frac{H_*((X \otimes_{\mathcal{R}_1}Y) / V)}{U\text{-torsion}}$ as free $\F[V]$ and $\F[U]$-modules respectively. It turns out that the possible pairs of ranks are $(1,1), (2, 0), (0,2), (4,0), (0,4)$, and $(0,0)$ corresponding to types $C$, $S_h$, $S_v$, $S_h\oplus S_h$, $S_v \oplus S_v$ and $0$ up to $\sim$.
\end{proof}
The main reason for our interest in Lemma \ref{lemma:sim relations} is that among all combinations of products of complexes of different types, there is only one way to obtain a standard complex -- via a tensor product of two standard complexes. This observation leads us to define an even more coarse equivalence relation $\approx$ on the set of finitely generated free chain complexes over $\mathcal{R}_1$.
\begin{definition}
Let $C_1$ and $C_2$ be finitely generated free chain complexes over $\mathcal{R}_1$. Define $C_1 \approx C_2$ if $C_1 \oplus D_1 \cong C_2 \oplus D_2$ for some direct sums of horizontal snake complexes, vertical snake complexes, and local systems $D_1$ and $D_2$.
\end{definition}
As $\sim$ above, $\approx$ is an equivalence relation and we think of its equivalence classes as multisets of standard complexes.
\begin{definition}
    Let $n \in \N_0$ and let $a_1, \dots, a_{2n}$ be a sequence of nonzero integers. Let $C = C(a_1, \dots, a_{2n})$ be a standard complex with the basis $\{x_0, \dots, x_{2n}\}$. Define
    \begin{itemize}[label = $-$]
        \item $\tau(C) = -\frac{1}{2}(\gr_V(x_0) - \gr_V(x_n))$,
        \item $\epsilon(C) = \mathrm{sign}(a_1)$, unless $n=0$ in which case $\epsilon(C)=0$, and 
        \item $\phi_j(C) = |\{ i \ | \ i \text{ odd and } a_i = j \}| - |\{ i \ | \ i \text{ odd and } a_i = -j \}|$.
    \end{itemize}
\end{definition}
These definitions might seem opaque and unmotivated at first, but they become much clearer when one draws a picture of a standard complex and interprets them geometrically. For example, starting at $x_0$ and moving towards higher generators, $\epsilon(C)$ denotes the direction of the first arrow and $\phi_7(C)$ denotes the signed number of horizontal arrows of length $7$. See also Section \ref{sec:example} for the calculation of these invariants for a particular standard complex corresponding to the torus knot $T_{3,4}$.

\begin{definition}\label{def:invariants tau, epsilon, phi}
    Let $K \subset S^3$ be a knot, $\CH_{\mathcal{R}_1}(K)$ its full Khovanov-Rozansky complex and $C_1, \dots, C_n$ the unique standard complexes such that $\CH_{\mathcal{R}_1}(K) \approx C_1 \oplus \dots \oplus C_n$. Define
    \begin{itemize}[label=$-$]
        \item $\tau^H(K) = \frac{1}{n}\sum_{i=1}^n \tau(C_i)$,
        \item $\epsilon^H(K) = \frac{1}{n}\sum_{i=1}^n \epsilon(C_i)$, and
        \item $\phi_j^H(K) = \frac{1}{n}\sum_{i=1}^n \phi_j(C_i)$ for all $j\in\N$.
    \end{itemize} 
\end{definition}
Note that the newly defined invariants take values in $\Q$ rather than $\Z$ as in knot Floer homology. This definition is essentially all of the content of Theorem \ref{thm:invariants}.
\begin{remark}
    We decided to define $\tau(C) = -\frac{1}{2}(\gr_V(x_0) - \gr_V(x_n))$ using $\gr_V$. This seems like an arbitrary choice and one could wonder if a new invariant is obtained by defining $\tau'(C) = -\frac{1}{2}(\gr_U(x_0) - \gr_U(x_n))$ instead. Indeed, on each individual standard complex $C$, the values $\tau(C)$ and $\tau'(C)$ generally differ. However, the Khovanov-Rozansky chain complex $\CH_{\mathcal{R}_1}(K)$ is symmetric and as such contains the summand $C(-a_n, \dots, -a_1)$ for each asymmetric summand $C(a_1, \dots, a_n)$. The effect of this symmetry on $\tau$ and $\tau'$ is that $\tau(C(a_1, \dots, a_n)) = -\tau'(C(-a_n, \dots, -a_1))$. Since $\tau^H(K)$ takes into account all standard complexes, we obtain $\tau^H(K) = -\tau'^H(K)$ and both invariants are in fact equivalent.
\end{remark}
Finally, we establish that the newly defined invariants behave nicely under the operation of connected sum.
\begin{proposition}\label{prop:additivity of invariants}
Let $\mathcal{C} = (\{ K \subset S^3 \}, \#)$ be the commutative monoid of knots under the operation of connected sum. Then $\tau^H$ and $\phi^H_j: \mathcal{C} \to \Q$ are monoid homomorphisms for all $j \in \N$. Equivalently, let $K_1, K_2 \subset S^3$ be knots. Then
\begin{itemize}[label = $-$]
    \item $\tau^H(K_1\#K_2) = \tau^H(K_1) + \tau^H(K_2)$ and
    \item $\phi_j^H(K_1\#K_2) = \phi_j^H(K_1) + \phi_j^H(K_2)$ for all $j \in \N$.
\end{itemize}
\end{proposition}
\begin{proof}
By Theorem \ref{thm:HOMFLY splitting}, let $C_1, \dots, C_n$ and $C_1', \dots, C_m'$ be the unique standard complexes such that $\CH_{\mathcal{R}_1}(K_1) \approx C_1 \oplus \dots \oplus C_n$ and $\CH_{\mathcal{R}_1}(K_2) \approx C_1' \oplus \dots \oplus C_m'$. By Lemma \ref{lemma:sim relations}, all standard complex summands of $\CH_{\mathcal{R}_1}(K_1\# K_2) \cong \CH_{\mathcal{R}_1}(K_1) \otimes_{\mathcal{R}_1} \CH_{\mathcal{R}_1}(K_2)$ arise from the tensor products $C_i \otimes_{\mathcal{R}_1} C_k'$ for some $i \in \{1, \dots, n\}$ and $k \in \{1, \dots, m\}$. There are $nm$ of them. The quantities $\tau(C)$ and $\phi_j(C)$ were proven to be additive by \cite[Theorem 7.2]{dai2021more}, \emph{i.e.} they satisfy $\tau(C_i \otimes_{\mathcal{R}_1} C_k') = \tau(C_i) + \tau(C_k')$ and $\phi_j(C_i \otimes_{\mathcal{R}_1} C_k') = \phi_j(C_i) + \phi_j(C_k')$ for all $j$. Let us first consider $\tau$ and note that $nm\ \tau^H(K_1 \# K_2)$ is by definition the sum of all $\tau(C_i\otimes_{\mathcal{R}_1} C_k')$. In other words, we can calculate $nm\ \tau^H(K_1 \# K_2) = \sum_{i,k} \tau(C_i \otimes C_k') = \sum_i m\tau(C_i) + \sum_k n\tau(C_k') = m\tau^H(K_1) + n\tau^H(K_2)$ using additivity of $\tau$. Dividing the equation by $nm$ yields $\tau^H(K_1\# K_2) = \tau^H(K_1) + \tau^H(K_2)$ as required. The only essential ingredient of this argument was that $\tau$ is additive, so replacing $\tau$ by $\phi_j$ everywhere establishes the claim for all $\phi_j^H$ as well.
\end{proof}
We can now deduce the last main result of this paper.
\begin{proof}[Proof of Theorem \ref{thm:invariants}]
The invariants $\tau^H$, $\epsilon^H$, and $\phi_j^H$ for all $j \in \N$ are defined in Definition \ref{def:invariants tau, epsilon, phi}. The fact that they are well-defined follows from uniqueness of the splitting in Theorem \ref{thm:HOMFLY splitting}. Proposition \ref{prop:additivity of invariants} shows that the invariants $\tau^H$ and $\phi_j^H$ are additive.
\end{proof}
It remains to be seen whether $\tau^H$, $\epsilon^H$, and $\phi^H_j$ for all $j \in \N$ are concordance invariants. This would mean that $\tau^H$ and $\phi_j^H$ are concordance homomorphisms.
\begin{remark}
The strategy used in this paper to define the Khovanov-Rozansky versions of $\tau$, $\epsilon$, and $\phi_j$ can be applied in knot Floer homology to extend them to null-homologous links in arbitrary $3$-manifolds. 
\end{remark}
\begin{remark}
If Rasmussen's $E(1)$ spectral sequence from $\widehat{\mathit{H}}(K)$ to $\mathfrak{sl}_1$ homology collapses, then we are in the setting of Proposition \ref{prop:HOMFLY width 1} and the new invariants are all determined by $\tau^H(K)$ via $\epsilon^H(K) = \mathrm{sign}(\tau^H(K))$, $\phi_1^H(K) = \tau^H(K)$, and $\phi^H_j(K) = 0$ for all $j \geq 2$.
\end{remark}

\section{Example}\label{sec:example}
This section contains the full version of Khovanov-Rozansky homology of the $T_{3,4}$ torus knot together with the calculations of invariants $\tau^H$, $\epsilon^H$ and $\phi_j^H$. The form of its reduced Khovanov-Rozansky homology together with higher differentials was conjectured in \cite[Figure 3.7]{dunfield2006superpolynomial} and proven in \cite{rasmussen2016some}. In our pictorial representation, its full Khovanov-Rozansky complex is depicted in Figure \ref{fig:T(3,4) complicated}. We can perform a change of basis over $\F[U,V]$ to transform it into the complex depicted in Figure \ref{fig:T(3,4) simple}. Note that this form consists of a standard complex $C(1, -1, 1, -1, 1, -1)$ and a single $1 \times 1$ square, which is in line with Proposition \ref{prop:HOMFLY width 1}. 
\begin{figure}[t]
    \begin{subfigure}[b]{0.48\textwidth}
        \centering
        \begin{tikzpicture}[scale=1.4]
            \draw[step=1.0,gray,thin] (-0.5,-3.5) grid (3.5,0.5);
            \def\a{0.08}
            \draw[black, very thick] (0,0) -- (1,0) -- (1,-1) -- (2-\a,-1-\a) -- (2, -2) -- (3, -2) -- (3, -3);
            \draw[black, very thick] (3,0) -- (2,0) -- (2+\a,-1+\a) -- (3,-1) -- (3,0);

            \draw[black, very thick] (0,0) to[out=15, in=165] (2,0){};
            \draw[black, very thick] (1,0) to[out=15, in=165] (3,0){};
            \draw[black, very thick] (1,-1) to[out=20, in=160] (3,-1){};
            
            \draw[black, very thick] (3,0) to[out=-75, in=75] (3,-2){};
            \draw[black, very thick] (3,-1) to[out=-75, in=75] (3,-3){};
            \draw[black, very thick] (2,0) to[out=-70, in=70] (2,-2){};

            \draw[black, very thick] (2,0) -- (1, -1) {};
            \draw[black, very thick] (2,-2) -- (3, -1) {};
            \draw[black, very thick] (3,0) to[out=200, in=85] (2-\a, -1-\a) {};

            \filldraw[black] (0,0) circle (2pt) node[anchor=north west]{};
            \filldraw[black] (1,0) circle (2pt) node[anchor=north west]{};
            \filldraw[black] (1,-1) circle (2pt) node[anchor=north west]{};
            \filldraw[black] (2-\a,-1-\a) circle (2pt) node[anchor=north west]{};
            \filldraw[black] (2,-2) circle (2pt) node[anchor=north west]{};
            \filldraw[black] (3,-2) circle (2pt) node[anchor=north west]{};
            \filldraw[black] (3,-3) circle (2pt) node[anchor=north west]{};
            
            \filldraw[black] (3,0) circle (2pt) node[anchor=north west]{};
            \filldraw[black] (2,0) circle (2pt) node[anchor=north west]{};
            \filldraw[black] (2+\a,-1+\a) circle (2pt) node[anchor=north west]{};
            \filldraw[black] (3,-1) circle (2pt) node[anchor=north west]{};
        \end{tikzpicture}
        \caption{}
        \label{fig:T(3,4) complicated}
    \end{subfigure}
    \begin{subfigure}[b]{0.48\textwidth}
        \centering
        \begin{tikzpicture}[scale=1.4]
            \draw[step=1.0,gray,thin] (-0.5,-3.5) grid (3.5,0.5);
            \def\a{0.08}
            \draw[black, very thick] (0,0) -- (1,0) -- (1,-1) -- (2-\a,-1-\a) -- (2, -2) -- (3, -2) -- (3, -3);
            \draw[black, very thick] (3,0) -- (2,0) -- (2+\a,-1+\a) -- (3,-1) -- (3,0);
            
            \filldraw[black] (0,0) circle (2pt) node[anchor=north west]{$x_0$};
            \filldraw[black] (1,0) circle (2pt) node[anchor=north west]{};
            \filldraw[black] (1,-1) circle (2pt) node[anchor=north west]{};
            \filldraw[black] (2-\a,-1-\a) circle (2pt) node[anchor=north west]{};
            \filldraw[black] (2,-2) circle (2pt) node[anchor=north west]{};
            \filldraw[black] (3,-2) circle (2pt) node[anchor=north west]{};
            \filldraw[black] (3,-3) circle (2pt) node[anchor=north west]{$x_6$};
            
            \filldraw[black] (3,0) circle (2pt) node[anchor=north west]{};
            \filldraw[black] (2,0) circle (2pt) node[anchor=north west]{};
            \filldraw[black] (2+\a,-1+\a) circle (2pt) node[anchor=north west]{};
            \filldraw[black] (3,-1) circle (2pt) node[anchor=north west]{};
        \end{tikzpicture}
        \caption{}
        \label{fig:T(3,4) simple}
    \end{subfigure}
    \caption{The full Khovanov-Rozansky complex $\CH(T_{3,4})$ is shown in (\textsc{a}). Through a change of basis, it can be shown to be isomorphic to the complex depicted in (\textsc{b}), which splits into a standard complex and a local system. This can be used to calculate invariants $\tau^H(T_{3,4})=3$, $\epsilon^H(T_{3,4})=1$, $\phi_1(T_{3,4})=3$, and $\phi_j(T_{3,4})=0$ for $j \geq 2$.}
    \label{fig:T(3,4) HOMFLY complex}
\end{figure}

The splitting of $\CH_{\mathcal{R}_1}(T_{3,4})$ into standard complexes and local systems can also be used to calculate the invariants $\tau^H$, $\epsilon^H$ and $\phi_j^H$. We have $\tau^H(\CH(T_{3,4})) = 3$. With notation as in Figure \ref{fig:T(3,4) simple}, this can be calculated by noting that $\gr_V(x_6)=0$ by Definition \ref{def:standard complex}, because $x_6$ is a horizontal homology generator. Together with $\gr_V(V)=-2$, and $\gr_V(\partial)=-1$, this is sufficient information to establish that $\gr_V(x_0) = -6$ and hence $\tau^H(\CH(T_{3,4})) = -\frac{1}{2} \gr_V(x_0) = 3$ as required. Similarly, we have $\epsilon^H(\CH(T_{3,4})) = 1$ since there is only one standard complex and it starts with a positive number. Finally, we have $\phi_1^H(\CH(T_{3,4})) = 3$ since there are three horizontal arrows of signed length $1$ and no horizontal arrows of signed length $-1$. There are no longer arrows so $\phi_j^H(\CH(T_{3,4}))=0$ for all $j \geq 2$.

\bibliography{mybib.bib}
\bibliographystyle{alpha}

\end{document}

%% file: mymacros.tex
\usepackage{amsmath}
\usepackage{amsfonts}
\usepackage{amssymb}
\usepackage{amsthm}
\usepackage{mathtools}
\usepackage{tikz}
\usepackage{tikz-cd}
\usepackage{hyperref}
\usepackage{enumitem}
\usepackage{graphicx}
\usepackage{subcaption}

\usetikzlibrary{arrows}
\usetikzlibrary{decorations.pathreplacing}
\usetikzlibrary{calligraphy}
\usetikzlibrary{external}

\tikzset{knot/.style={black,very thick,preaction={draw,white,line width=5pt}}}

\colorlet{dark green}{green!50!black}

\DeclareMathOperator{\ad}{ad}

\newcommand{\F}{\ensuremath{\mathbb{F}} }
\newcommand{\N}{\ensuremath{\mathbb{N}} }
\newcommand{\Z}{\ensuremath{\mathbb{Z}} }
\newcommand{\Q}{\ensuremath{\mathbb{Q}} }

\newcommand{\gr}{\mathrm{gr}}
\newcommand{\CFK}{\mathit{CFK}}
\newcommand{\HFK}{\mathit{HFK}}

\newcommand{\HF}{\mathit{HF}}

\newcommand{\CH}{\mathit{CH}}

\newtheorem{theorem}{Theorem}[section]

\newtheorem{proposition}[theorem]{Proposition}
\newtheorem{conjecture}[theorem]{Conjecture}

\newtheorem{lemma}[theorem]{Lemma}

\theoremstyle{definition}
\newtheorem{definition}[theorem]{Definition}

\theoremstyle{remark}
\newtheorem*{remark}{Remark}
\newtheorem{example}[theorem]{Example}

%% file: images/finite_CH.tex
\begin{figure}[t]
    \centering
    \begin{subfigure}{0.20\textwidth}
        \centering
        \begin{tikzpicture}[scale=1.5]
            \draw[step=1.0,gray,thin] (-0.5,-0.5) grid (1.5,1.5);
            \draw[black, very thick] (1,0) -- (0,0) -- (0,1);
            \filldraw[black] (0,0) circle (2pt) node[anchor=north west]{$b$};
            \filldraw[black] (1,0) circle (2pt) node[anchor=north west]{$a$};
            \filldraw[black] (0,1) circle (2pt) node[anchor=north west]{$c$};
        \end{tikzpicture}
        \caption{}
        \label{fig:CH(trefoil)}
    \end{subfigure}
    \hspace{1cm}
    \begin{subfigure}{0.20\textwidth}
        \centering
        \begin{tikzpicture}[scale=1.5]
            \draw[step=1.0,gray,thin] (-0.5,-0.5) grid (1.5,1.5);
            \def\a{0.12}
            \draw[black, very thick] (1,0) -- node[below]{$-1$} (0,0) -- (0,1) -- (1,1) -- (1,0);
            \filldraw[black] (0,0) circle (2pt) node[anchor=north east]{$b$};
            \filldraw[black] (1,0) circle (2pt) node[anchor=north west]{$c$};
            \filldraw[black] (0,1) circle (2pt) node[anchor=south east]{$a$};
            \filldraw[black] (1,1) circle (2pt) node[anchor=south west]{$d$};
            \filldraw[black] (1-\a,1-\a) circle (2pt) node[anchor=north east]{$e$};
        \end{tikzpicture}
        \caption{}
        \label{fig:CH(figure 8 knot)}
    \end{subfigure}
    \caption{The Khovanov-Rozansky chain complex $\CH(T_{2,-3})$ is depicted in ($\textsc{a}$). The differential satisfies $\partial a = Ub$ due to the horizontal arrow of length $1$ connecting $a$ and $b$. Similarly, $\partial c = Vb$ due to the vertical arrow of length $1$ connecting $c$ and $b$. The absolute grading information cannot be inferred from the picture, but the relative grading information is preserved on each connected component. For example, since $\gr(U) = (-2,0,0)$ and $\gr(\partial) = (-1,-1,-2)$, one can deduce that $\gr(a)-\gr(b) = \gr(U)-\gr(\partial) = (-1,1,2)$. The full Khovanov-Rozansky complex $\CH(4_1)$ is shown in ($\textsc{b}$). We have $\partial c = -Ub$ since the arrow from $c$ to $b$ is decorated with a $-1$.}
    \label{fig:finite CH}
\end{figure}

%% file: images/22.tex
\begin{figure}[t]
    \centering
    \begin{tikzpicture}[scale=1.3]
            \def\a{0.10}
            \draw[step=1.0,gray,thin] (0.5,0.5) grid (3.5,3.5);
            \draw[very thick, dashed] (3,3) to["$d_{1|0}$", swap] (2,3);
            \draw[very thick, dashed] (2,3) -- (1,1) node[yshift=20, xshift=-1]{$d_{1|2}$};
            \filldraw[black] (2,3) circle (2pt) node[anchor=north west]{$a$};
            \filldraw[black] (1,3) circle (2pt);
            \filldraw[black] (1,2) circle (2pt);
            \filldraw[black] (2,2) circle (2pt);
            \filldraw[black] (3,2) circle (2pt);
            \filldraw[black] (2,1) circle (2pt);
            \draw[fill=white] (1,1) circle (2pt) node[anchor=north west]{$y$};
            \draw[fill=white] (3,3) circle (2pt) node[anchor=north west]{$x$};
            \filldraw[black] (3,1) circle (2pt) node[anchor=north west]{$m_1=2$};
            \draw[very thick] (1-\a, 3-\a) to[out=135, in=225] (1-\a, 3+\a);
            \draw[very thick] (1-\a, 3+\a) to[out=45, in=135] (1+\a, 3+\a);
            \draw[very thick] (3-\a, 1-\a) to[out=315, in=225] (3+\a, 1-\a);
            \draw[very thick] (3+\a, 1-\a) to[out=45, in=315] (3+\a, 1+\a);
            \draw[very thick] (1-\a, 3-\a) to (3-\a,1-\a);
            \draw[very thick] (1+\a, 3+\a) to (3+\a,1+\a);
        \end{tikzpicture}
    \caption{Terms in Example \ref{example:22} correspond to black dots. The proof shows that the three terms on the diagonal $m_1=2$ sum to $0$ and the remaining four terms from diagonals $m_1=1$ and $m_1=3$ sum to $0$.}
    \label{fig:22}
\end{figure}

%% file: images/example_of_a_standard_complex.tex
\begin{figure}[t]
    \centering
    \begin{subfigure}[b]{0.25\textwidth}
        \centering
        \begin{tikzpicture}[scale=0.8]
            \draw[step=1.0,gray,thin] (-0.5,1.5) grid (2.5,4.5);
            \draw[black, very thick] (0,2) -- (1,2) -- (1,4) -- (2,4);
            \filldraw[black] (0,2) circle (2pt) node[anchor=north west]{$x_0$};
            \filldraw[black] (1,2) circle (2pt) node[anchor=north west]{$x_1$};
            \filldraw[black] (1,4) circle (2pt) node[anchor=north west]{$x_2$};
            \filldraw[black] (2,4) circle (2pt) node[anchor=north west]{$x_3$};
        \end{tikzpicture}
        \caption{}
        \label{fig:example of a horizontal snake cx}
    \end{subfigure}
    \begin{subfigure}[b]{0.25\textwidth}
        \centering
        \begin{tikzpicture}[scale=0.8]
            \draw[step=1.0,gray,thin] (-0.5,1.5) grid (2.5,4.5);
            \draw[black, very thick] (0,2) -- (0,3) -- (2,3) -- (2,4);
            \filldraw[black] (0,2) circle (2pt) node[anchor=north west]{$x_0$};
            \filldraw[black] (0,3) circle (2pt) node[anchor=north west]{$x_1$};
            \filldraw[black] (2,3) circle (2pt) node[anchor=north west]{$x_2$};
            \filldraw[black] (2,4) circle (2pt) node[anchor=north west]{$x_3$};
        \end{tikzpicture}
        \caption{}
        \label{fig:example of a vertical snake cx}
    \end{subfigure}
    \begin{subfigure}[b]{0.35\textwidth}
        \centering
        \begin{tikzpicture}[scale=0.8]
            \draw[step=1.0,gray,thin] (-0.5,1.5) grid (4.5,4.5);
            \draw[black, very thick] (0,4) -- (2,4) -- (2,2) -- (1,2) -- (1,3) -- (4,3) -- (4,2);
            \filldraw[black] (0,4) circle (2pt) node[anchor=north west]{$x_0$};
            \filldraw[black] (2,4) circle (2pt) node[anchor=north west]{$x_1$};
            \filldraw[black] (2,2) circle (2pt) node[anchor=north west]{$x_2$};
            \filldraw[black] (1,2) circle (2pt) node[anchor=north west]{$x_3$};
            \filldraw[black] (1,3) circle (2pt) node[anchor=north west]{$x_4$};
            \filldraw[black] (4,3) circle (2pt) node[anchor=north west]{$x_5$};
            \filldraw[black] (4,2) circle (2pt) node[anchor=north west]{$x_6$};
        \end{tikzpicture}
        \caption{}
        \label{fig:example of a standard cx}
    \end{subfigure}
    \caption{A horizontal snake complex $S_h(1, 2, 1)$ is shown in ($\textsc{a}$), a vertical snake complex $S_v(1,2,1)$ in ($\textsc{b}$), and the standard complex $C(2, -2, -1, 1, 3, -1)$ in ($\textsc{c}$).}
    \label{fig:example complexes}
\end{figure}